\documentclass[11pt,bezier]{article}
\usepackage{amsmath,amssymb,amsfonts}
\usepackage{hyperref}
\usepackage{graphicx} % inserting images
\usepackage{float}

\newtheorem{thm}{Theorem}[section]
\newtheorem{defi}[thm]{Definition}

\newtheorem{remark}[thm]{Remark}
\newtheorem{cor}[thm]{Corollary}

\newtheorem{example}[thm]{Example}

\newenvironment{proof}[1]{{\bf proof. }{\rm #1}}{\hfill $\rule {2mm}{2mm}$\\}

\begin{document}

 \title{A Completion of the spectrum of 3-way $(v,k,2)$ Steiner trades \\ }
\author{
{\sc Saeedeh Rashidi\footnote{Mahani Mathematical Research Center,
		Shahid Bahonar University of Kerman, Kerman, Iran} AND Nasrin Soltankhah}
\\[5mm]
Department of Applied Mathematics,\\
Faculty of Mathematics and Computer,\\
Shahid Bahonar University of Kerman,\\ Kerman,\ Iran\\
{saeedeh.rashidi@uk.ac.ir}\\
Department of Mathematics\\ Alzahra University \\
Vanak Square 19834 \ Tehran, I.R. Iran \\
{ soltan@alzahra.ac.ir}\\}
\date{\bf{ }}

% ---------------------------------------------------------------------------------------------------------------------------------------------------------------------------------------
\maketitle

 % -----------------------------------------------------------------------------------------------------------------------------------------------------------------------------------

 \begin{abstract}
A 3-way $(v,k,t)$ trade $T$ of volume $m$ consists of three pairwise disjoint collections
$T_1$, $T_2$ and $T_3$, each of $m$ blocks of size $k$, such that for
every $t$-subset of $v$-set $V$, the number of blocks containing this
$t$-subset is the same in each $T_i$ for $1\leq i\leq 3$.
If any $t$-subset of found($T$) occurs at most once
in each $T_i$ for $1\leq i\leq 3$, then $T$ is
called 3-way $(v,k,t)$ Steiner trade.\\
We attempt to complete the spectrum $S_{3s}(v,k)$, the set of  all possible volume sizes, for 3-way $(v,k,2)$ Steiner trades, by applying
some block designs, such as BIBDs, RBs, GDDs, RGDDs, and $r\times s$ packing grid blocks.\\ 
Previously,  we obtained some results about the existence some 3-way $(v,k,2)$ Steiner trades.
In particular, we proved that there exists a 3-way $(v,k,2)$ Steiner trade of
volume $m$ when $12(k-1)\leq m$ for $15\leq k$ (Rashidi and Soltankhah, 2016). Now, we show that the claim is correct also for $k\leq 14$.
\end{abstract}
\hspace*{-2.7mm} {\bf MSC:} {\sf 05B30; 05B05}\\
\hspace*{-2.7mm} {\bf KEYWORDS:} { \sf 3-way $(v,k,2)$ Steiner trade; 1-solely balanced set; Resolvable block design; Resolvable GDD; Resolvable  $r\times c$ grid-block packing }
% ---------------------------------------------------------------------------------------------------------------------------------------------------------------------------------------------------------------------------------------
\maketitle
% ----------------------------------------------------------------------------------------------------------------------------------------------------------------------------------------------------------------------------------------............
%%%%%%%%%%%%%%%%%%%%%%%%%%%%%%%%%%%%%%%%%%%%%%%%%%%%%%%%%%%%%%%%%%%%%%%%%%%%%%%%%%%%%%%%%%%
%-----------------------------------------------------------------------------------------------------------------------------------------------------------------------------------------------------------------------------------------
\section{Introduction}
 Let $k$, $t$ and $\lambda$ be positive integers such that $v>k>t$. A $t-(v,k,\lambda)\ design$ $({V},{B})$ is a collection of blocks such that each $t$-subset of $v$-set $V$ is contained in $\lambda$ blocks of collection $B$.
 A $(v,k,t)$ $trade$ $T=\{T_1,T_2\}$ of volume $m$ consists of two disjoint collections $T_1$ and $T_2$,
 each containing $m$ blocks. Each $block$ is a $k$-subset of $v$-set $V$.
 Every $t$-subset of $v$-set $V$ is contained in the same number of blocks in $T_1$ and $T_2$.
 In a $(v, k, t)$ trade, both collections of blocks must cover the same set of elements. This set of elements is $foundation$
 of $(v,k,t)$ trade and is denoted by found$(T)$.
 A $(v,k,t)$ trade is called $(v,k,t)$ $Steiner$ $trade$ if any $t$-subset of found$(T)$ occurs at most
 once in $T_1(T_2)$.\\
  Recently, a generalization for the concept of trade as has been defined in \cite{3} by the name of  $\mu$-way $(v,k,t)$ trade as follows:\\
 A	 $\mu$-$way$ $(v,k,t)$ $trade$ of volume $m$ consists of $\mu$ pairwise disjoint collections
 	$T_1,\ldots,T_{\mu}$ each of $m$ blocks,
 	such that for every $t$-subset of $v$-set $V$, the number of blocks containing this
 	$t$-subset is the same in each $T_i$ for $1\leq i\leq \mu$.
 	In the other words any set $T=\{T_i,T_j\}$ for $i\ne j$ is a $(v,k,t)$ trade of volume $m$.
 	A $\mu$-way $(v,k,t)$ trade is $\mu$-$way$ $(v,k,t)$ $Steiner\ trade$
 	if any $t$-subset of found$(T)$ occurs at most once in every $T_j$ for $j\geq 1$.
The concept of $\mu$-way trade was defined under a  different name as $N$-legged trade, before, see \cite{forbes}. 
    Each $T_i$ contains $m$ blocks $B_{1i},\ B_{2i},\cdots B_{mi}$, where $B_{ij}$ denote $i$th block of $T_j$. Therefore, a $\mu$-way trade $T$ of volume $m$ has the following form:
    	$$
    	\begin{tabular}{c|c|c|c|c|c}
    	$T_1$ & $T_2$ & $\cdots$ & $T_i$ & $\cdots$ & $T_{\mu}$\\
    	\hline
    	$T_{11}$ & $T_{12}$ & $\cdots$ & $T_{1i}$ & $\cdots$ & $T_{1\mu}$\\
    	$T_{21}$ & $T_{22}$ & $\cdots$ & $T_{21}$ & $\cdots$ & $T_{1\mu}$\\
    	$\cdots$ & $\cdots$ & $\cdots$ & $\cdots$ & $\cdots$ & $\cdots$\\
    	$T_{m1}$ & $T_{m2}$ & $\cdots$ & $T_{mi}$ & $\cdots$ & $T_{m\mu}$
    	\end{tabular}
    	$$ 
    	 A type of $\mu$-way $(v,k,t)$ Steiner trade with additional property is named 
    	 $\mu$-way $t$-solely balanced set where those sets have the important role in
    	 constructing the Steiner trade with other parameters.  
\begin{defi}
	Let $T=\{T_1,\ldots,T_{\mu}\}$ be a $\mu$-way $(v,k,t)$ Steiner trade. It is
	$\mu$-$way$ $t$-$solely$ $balanced$ $set$ if there exist no blocks $B_{ij}$ and $B_{ ab }$ such that
	$|B_{ij}\cap B_{ab }|>t$ for $1\leq j<b\leq \mu$ and $1\leq i,a\leq m$. In the other words $B_{ij}$ and $B_{ab}$ contain no common $(t+1)-$subset.
\end{defi}
For $t=1$ there exists a $\mu$-way 1-solely balanced set for $k=2$ and $v=2m, \mu \leq 2m-1$ of volume $m$. 
It is a one factorization of complete graph $K_{2m}$.
For $t\geq 2$ each large set of a super simple design or a resolvable super simple design
is a $\mu$-way $t$-solely balanced set.
Hamilton and Khodkar~\cite{11} named 2-solely balanced set with $strong$ $Steiner$ $trades$ that is
a $(v, k,2)$ Steiner trade $T=\{T_1,T_2\}$ so that any block of $T_1$ intersects any block of $T_2$ in at most two elements. 
Gray and Ramsay~\cite{1} applied $2$-way 1-solely balanced set for constructing 2-way $(v,k,t)$ Steiner trades.
These sets are named solely sets by them.\\
There exists the following main construction, which uses the $\mu$-way 1-solely balanced set $T$ and
$\mu$ new elements $\{x_1,\ldots,x_{\mu}\}$ for constructing 
the $\mu$-way $(v+\mu,k+1,t+1)$ Steiner trade $T^*$.  
\begin{center}
	\begin{tabular}{c|c|c|c}
		$T^*_{1}$& $T^*_{2}$&$\dots$ & $T^*_{\mu}$\\
		\hline
		$x_1T_1$ & $x_1T_2$& $\dots$ & $x_1T_{\mu}$\\
		$x_2T_2$& $x_2T_3$ & $\dots$& $x_2T_{1}$\\
		$x_3T_3$& $x_3T_4$ & $\dots$& $x_3T_2$\\
		$\vdots$& $\vdots$ & $\vdots$ &  $\vdots$\\
		$x_{\mu}T_{\mu}$& $x_{\mu}T_1$ & $\dots$& $x_{\mu}T_{\mu-1}$\\
	\end{tabular}
\end{center} 
This construction is mentioned in the following theorem.
\begin{thm}~\label{2}
	\cite{3}(i) Let $T=\{T_1,\ldots,T_{\mu}\}$ be a $\mu$-way $(v,k,t)$ trade of volume $m$.
	Next based on $T$, a $\mu$-way $(v+\mu,k+1,t+1)$ trade $T^*$ of volume $ \mu m$ can be constructed.\\
	(ii) If $T$ is  a $\mu$-way $t$-solely balanced, then a $\mu$-way $(v+\mu,k+1,t+1)$ Steiner trade $T^*$ can be constructed.
\end{thm}
We need some notation.
Let $T$  and $T^{*}$ be two $\mu$-way $(v,k,t)$ trades of volume $m$.\\
We consider $T+T^{*}=\{T_1\cup T_1^{*},\ldots,T_{\mu}\cup T_{\mu}^{*}\}$. It is easy to
see that $T+T^{*}$ is a
$\mu$-way $(v,k,t)$ trade. If $T$ and $T^{*}$ are Steiner trades and $\rm {found(T)}\cap \rm{found(T^*)}=\phi$,
then $T+T^{*}$ is also a Steiner trade.  
\\
 There exist many questions concerning $\mu$-way trades. Some of the most important questions are about the minimum volume and minimum foundation size and the set of all possible volume sizes of  $\mu$-way trades. Not much is known for the mentioned questions about $\mu$-way $(v,k,t)$ trades for $\mu\geq 3$ and most of the papers have been focused on the case $\mu=2$.
  Some questions have been answered about the existence and non-existence of $3$-way $(v,k,t)$ trades for some special values of $k$ and $t$, see \cite{61,62,3}.  
Let $\mathcal{S}_{3s }(t,k)$ denote the set of all possible volume sizes of a $3$-way $(v,k,t)$ Steiner trade. The set of all possible volume sizes of a $2$-way $(v,k,2)$ Steiner trade has been answered completely in~\cite{1, 5, 8}. 

The aim of this paper is to complete the lower bound for the spectrum of 3-way $(v,k,2)$ Steiner trades. 
Previously in~\cite{0}, we obtained some results about the existence some 3-way $(v,k,2)$ Steiner trades.
In particular, we proved that there exists a 3-way $(v,k,2)$ Steiner trade of
volume $m$ when $12(k-1)\leq m$ for $15\leq k$.
Here, we show it is correct also for $k\leq 14$.
Also we improve the lower bound to $11(k-1)$ for even $k$.
The obtained results in~\cite{0} are as follows:
\begin{thm}
	~\cite{0} $m\in S_{3s}(2,k) $ for all $m\geq12(k-1)$ when $k\geq 15$.
\end{thm}
Also there are some other results.
\begin{thm}~\cite{0}~\label{3}
	There exists a 3-way $(v,k,2)$ Steiner trade of volume $9(k-1)-r$ for $r\in\{1,\ldots,k-1\}$
	with block size $k$.
\end{thm}
\begin{thm}~\cite{0}~\label{4}
	(1) There exists a $\mu$-way $(q^2+\mu,q+1,2)$ Steiner trade of volume $m=q\mu$ for $\mu\in\{2,\ldots,q+1\}$
	when $q$ is a prime power.\\
	(2) $S_{3s}(2,k)\subseteq \mathcal{N}\setminus\{1,\ldots,3k-4\}$ for $k\ne 4$.\\
	(3) There exists a 3-way $(v,k,2)$ Steiner trade of volume:\\
	\hspace*{0.25cm} (a) $m=3l$ for $l\geq k-1$;\\
	\hspace*{0.25cm} (b) $m=4l$ for $l=k-1$ or $l\geq 2(k-1)$ when $l$ is odd;\\
	\hspace*{0.25cm} (c) $m=4l$ for all $l\geq 4(k-1)$.
\end{thm}
But we obtain three new constructions that will be introduced in the next
section for this aim. 

%%%%%%%%%%%%%%%%%%%%%%%%%%%%%%%%%%%%%%%%%%%%%%%%%%%%%%%%%%%%%%%%%%%%%%%%%%%%%%%%%%%%%%%%%%%
\section{New constructions}
In this section, we propose two constructive methods for 3-way $(v,k,t)$ Steiner trades.
In these constructions, we first have an 1-solely balanced set and then apply the
Theorem~\ref{2} for constructing the Steiner trades.
\subsection{Construction 1(GDD and RB)}
The main combinatorial objects in this section are GDDs and RBs.
\begin{defi}
	Suppose $(\mathcal{V},\mathcal{B})$ is a $2-(v,k,\lambda)$ Design. 
	A parallel class in $(\mathcal{V},\mathcal{B})$ is
	a subset of pairwise disjoint blocks from $\mathcal{B}$ whose union is $\mathcal{V}$.
	A partition of $\mathcal{B}$ into $r$ parallel
	classes is called a resolution, and $(\mathcal{V},\mathcal{B})$
	is said to be a $resolvable$ if $\mathcal{B}$ has
	at least one resolution. A resolvable $2-(v,k,\lambda)$ Design is named RB or $RB(v,k,\lambda)$.
\end{defi}
A parallel class contains $\frac{v}{k}$ blocks. Therefore, a $2-(v,k,\lambda)$ Design
can have a parallel class only if $v \equiv~0\ (\rm{mod}\ k)$. Also a RB$(v,k,\lambda)$ has $r=\frac{\lambda(v-1)}{k-1}$ parallel classes.
\begin{defi} 
	Let $K$ be a set of positive integers. A group divisible design $K$-GDD
	(as GDD for short) is a triple $(\mathcal{X}, \mathcal{G}, \mathcal{A})$
	satisfying the following properties:\\
	(1) $\mathcal{G}$ is a partition of a finite set $\mathcal{X}$ into subsets (called groups);\\
	(2) $\mathcal{A}$ is a
	set of subsets of $\mathcal{X}$ (called blocks), each of cardinality from $K$, such that a
	group and a block contain at most one common element;\\ (3) every pair of
	elements from distinct groups occurs in exactly one block.\\
	If $\mathcal{G}$ contains $u_i$ groups of size $g_i$, for $i\in\{1,\ldots,s\}$, then we denote
	by $g_1^{u_1}g_2^{u_2}\cdots g_s^{u_s}$
	the group type (or type) of the GDD. If $K = \{k\}$, we
	write $\{k\}$-GDD as $k$-GDD. 
This is exponential notation for the
group type. 
\end{defi}
	A GDD is resolvable if the blocks of it can be partitioned
	into parallel classes. A resolvable GDD is denoted by RGDD.
	Let $r^*$ be the number of parallel classes of the blocks of RGDD.
\begin{example}~\label{rr}
	The collection $D$ is a $RB(9,3,1)$. The blocks are
	written as columns. The blocks of one parallel class can be the groups. Hens the collection $D$ is a 3-GDD of type $3^3$ with $r^*=3$.
		\vspace{-0.5cm}
	\begin{center}
	   $D:$     	1\ 4\ 7\ \ \ \  1\ 2\ 3\ \ \ \   1\ 2\ 3\ \ \ \   1\ 2\ 3\\
                   \hspace{0.4cm} \ 2\ 5\ 8\ \ \ \   4\ 5\ 6\ \ \ \  5\ 6\ 4\ \ \ \   6\ 4\ 5\\
	               \hspace{0.4cm}	3\ 6\ 9\  \ \ \  7\ 8\ 9\ \ \ \  9\ 7\ 8\  \ \ \  8\ 9\ 7
	\end{center}
\end{example}
Now, we construct new 1-solely balanced sets and related Steiner trades using some RBs or GDDs.
%The following theorem is found 
%in the PBDs, Frames, and Resolvability section of ~\cite{6}, pages 261-265.
\begin{thm}~\cite{6}~\label{10}
	 The necessary conditions for the existence of a 4-RGDD of type
	 $g^u$, namely, $u\geq 4$, $gu\equiv \ 0\ (\rm{mod}\ 4)$ and $g(u-1)\equiv\  0\ ({\rm{mod}}\ 3)$
	 except $(g,u)\in\{(2,4),(3,4),(6,4),(2,10)\}$ and possibly excepting: $ g =2$
	 and $u \in\{34, 46, 52, 70, 82, 94, 100, 118, 130, 142, 178, 184, 202, 214, 238,\\ 250, 334, 346\}$;
	 $g = 10$ and $u \in\{4, 34, 52, 94\};$ $g \in[14, 454] \cup \{478, 502, 514, 526, 614, 626, 686\}$ and
	 $u \in\{10, 70, 82\}$; $g = 6$ and $u \in\{6, 54, 68\}$; $g = 18$ and $u \in\{18, 38, 62\}$; $g = 9$ and
	 $u = 44$; $g = 12$ and $u = 27$; $g = 24$ and $u = 23$; and $g = 36$ and $u \in\{11, 14, 15, 18, 23\}$.
\end{thm} 
\begin{remark}~\label{11}
	The number of blocks of $k$-RGDD, denoted $b$, is obtained from the following equations.
	$$ v=(k-1)r^*+g$$
	$$(r^*+1)v=kb+gu$$
	We explain this equation for the Example~\ref{rr}.
	Consider the group 147. The elements of this block appear in the three blocks of each parallel class of GDD.
	For example, consider the first parallel class.
	 The element 1 appears with the elements 2 and 3. The element 4 appears with the elements 5 and 6. The element 7 appears with the elements 8 and 9. 
	Hence $v=3+2+2+2=3+2\times 3=g+r^*(k-1)$.\\
	   The second equation: Each element occur exactly once in each parallel class and exactly once in the groups.
	  Therefore, $v(r^*+1)=gu+bk$.   
\end{remark}
\begin{thm}
	There exists a 3-way $(23,5,2)$ Steiner trade of volume 25.	
\end{thm}
\begin{proof}
	By Theorem~\ref{10} and Remark~\ref{11}, there exists a 4-RGDD of type $5^4$ and 
	$$20=3r^*+5$$
	$$(r^*+1)(3r^*+5)=4b+20.$$
	Therefore, $b=25$, $m=\frac{v}{k}=\frac{gu}{4}=\frac{20}{4}=5$ and $\mu=\frac{b}{m}=\frac{25}{5}$.
	Therefore, we have a 5-way 1-solely balanced set with volume $m=5$, foundation $v=20$ and block size $k=4$.
	By, Theorem~\ref{2}, we can construct a 3-way $(23,5,2)$ Steiner trade of volume $15$, $20$, and $25$. 
\end{proof}
The necessary conditions for the existence of
a $RB(v, k, \lambda)$ is:\\
1. $\lambda (v-1) \equiv\ 0\ ({\rm{mod}}\ k-1)$,\\
2. $v \equiv\ 0\ ({\rm{mod}}\ k)$.\\
The necessary conditions are sufficient for any $k$ and $\lambda$ if $v$ is large
enough. There exist four related theorems, for $\lambda=1$, which are as follows.
\begin{thm}~\cite{10}~\label{9}
	There exists a $RB(v,4,1)$ if and only if $v\equiv \ 4\ ({\rm{mod}}\ 12)$.	
\end{thm}
\begin{thm}~\cite{10}
	If $v$ and $k$ are both powers of the same prime, then the necessary
	conditions for the existence of an $RB(v, k, \lambda)$ are sufficient.
\end{thm}
\begin{thm}~\cite{10}
	Let $q=k(k-1)+1$ be a prime power, odd number and $q>[\frac{k(k-1)}{2}]^{k(k+1)}$, there exists a $RB(kq, k, 1)$.	
\end{thm}
\begin{thm}~\cite{10}
	For $k\geq 3$ and $q\equiv\ k \ ({\rm{mod}} \ k(k-1))$ and $v>exp(e^{12k^2})$, there exists a $RB(v, k, 1)$.
\end{thm} 
\begin{thm}~\label{111}
	The existence of a $RB(v,k,1)$ design is equivalent to a $\mu$-way 1-solely balanced set with block size $k$, volume
	$m=\frac{v}{k}$ and $\mu=\frac{v-1}{k-1}$.	 
\end{thm}
\begin{proof}
	Each $RB(v,k,1)$ has $\mu=\frac{v-1}{k-1}$ parallel classes, such as $P_1,\ldots, P_{\mu}$.
	Also each class contains $m=\frac{v}{k}$ blocks. Let $T_i=P_i$ for $1\leq i\leq \mu$.
	Now $T=\{T_1,\ldots,T_{\mu}\}$ is a $\mu$-way 1-solely balanced set with block size $k$, volume
	$m=\frac{v}{k}$ and $\mu=\frac{v-1}{k-1}$.	 
\end{proof}
 Now, we apply Theorem~\ref{2} and~\ref{111}, then there exists a $\mu$-way $(v+\mu,k+1,2)$ Steiner trade
 for large $v$ and volume $\mu m$. This confirms our results in~\cite{0}. 
\begin{thm}
	If there exists a $\mu$-way 1-solely balanced set with block size $k$, volume $m$ and foundation size $v$,
	such that $k=\frac{v}{m}$, then $\mu\leq\frac{v-1}{k-1}$. The equation holds when there exists a $RB(v,k,1)$.	
\end{thm}
\begin{proof}
	The number of pairs which appear in the $\mu$-way 1-solely balanced set with block size $k$ and volume $m$, is 
	$C(k,2).m.\mu=C(k,2).\frac{v}{k}.\mu$. By the property of 1-solely balanced sets (Each pair appears in at most one block.):
	$$C(k,2).\frac{v}{k}.\mu\leq C(v,2)=\frac{v(v-1)}{2}\Rightarrow \mu\leq\frac{v-1}{k-1}$$
	Now, if $\mu =\frac{v-1}{k-1}$, then the $\mu$-way 1-solely balanced set is a $RB(v,k,1)$.
\end{proof}
\begin{thm}~\label{7}
	There exists a 3-way $(v+2,5,2)$ Steiner trade of volume $\frac{\mu(v-1)}{4}$,
	for $3\leq \mu\leq \frac{v-2}{3}, \ v\geq11$
	and
	$v\equiv\ 5 \ ({\rm{mod}}\  12)$.
\end{thm}
\begin{proof}
	By Theorem~\ref{9}, there exist $RB(v-1,4,1)$ for $v-1\equiv 4\ ({\rm{mod}}\ 12)$.
	Therefore, by Theorem~\ref{111} there exists a $\mu$-way 1-solely balanced set with block size $k=4$,
	volume $m=\frac{v-1}{4}$ and $3\leq \mu\leq \frac{v-2}{3}$.
	Now, we apply Theorem~\ref{2} and obtain a 3-way $(v+2,5,2)$ Steiner trade of volume $\mu m$ for $m=\frac{v-1}{4}$ and 
	$3\leq \mu\leq \frac{v-2}{3}.$
\end{proof}	
\subsection{Construction 2 ( Resolvable  $r\times c$ grid-block packing)}
In this subsection, we obtain some $\mu$-way 1-solely balanced sets by resolvable  $r\times c$ grid-block packings and vice versa.
\begin{defi}
	\cite{7}For a $v$-set $V$, let ${\cal A}$ be a collection of $r \times c$ arrays with elements in $V$. A pair $(V, {\cal A})$ is called an {\it $r\times c$ grid-block packing} (an {\it $r\times c$ grid-block design}) if any two distinct points $i$ and $j$ in $V$ occur together at most once (exactly once) in the same row or in the same column of arrays in ${\cal A}$. A $r\times c$ grid-block packing is denoted by $P_{r\times c}(K_{v})$  ($D_{r\times c}(K_{v})$).
	\end{defi}
	\begin{remark}
	Notice that by this definition the following table cannot be a grid-block of an $r\times c$ grid-block packing.
	Since the pair $\{1,4\}$ occurs in the first column and the second row.
	Therefore, the grid-block cannot contain the repetitive element.
	\begin{center}
		\begin{tabular}{ccc}
		$1$&$2$&$3$\\
		$4$&$5$&$1$\\
		$6$&$1$&$7$
		\end{tabular}
	\end{center} 
	\end{remark}
\begin{defi}	
	 \cite{7} An $r\times c$ grid-block packing $(V, {\cal A})$ is said to be {\it resolvable} if the collection of arrays ${\cal A}$ can be partitioned into sub-classes $\boldsymbol{R}_1$,$\ldots$,$\boldsymbol{R}_t$ such that every point of $V$ is contained in precisely one array of each class.\\
	 If a packing is resolvable, then $v$ is divisible by $rc$ and the number
	of grid-block is $b=t\frac{v}{rc}\leq \lfloor \frac{v-1}{r+c-2}\rfloor$, where $t$ is the number of
	resolution class. A resolvable packing $P_{r\times c}(K_v)$ attaining  this bound is said to be optimal. 
\end{defi}
There exist some existence results about  $P_{q\times q}(K_{q^n})$ and  $D_{q\times q}(K_{q^n})$.
We apply the following theorems in this section.
\begin{thm}~\cite{7}~\label{31}
	An optimal resolvable grid-block packing $P_{q\times q}(K_{q^n})$
	exists for a prime power $q$ and an integer $n$. Moreover, when $n$ is even 
	and $q$ is odd, the optimal resolvable grid-block packing is a resolvable 
	grid-block design $D_{q\times q}(K_{q^n})$.
\end{thm}
The aim of $PA(v,c,1)$ is a resolvable packing however there does not exist any definition of it in~\cite{7}.
\begin{thm}~\cite{7}~\label{32}
	Assume $r\leq c$. If there exists a resolvable $PA(v,c,1)$ with $t$
	resolution classes and a resolvable $P_{r\times c}(K_{rc})$ with $s+1$ grid-blocks,
	then there exists a resolvable $P_{r\times c}(K_{rv})$ with $st+1$ resolution classes.
\end{thm}
\begin{thm}~\cite{7}~\label{33}
	Assume $r\leq c$. If there exists a resolvable $PA(v,c,1)$ with $t$
	resolution classes,
	then there exists a resolvable $P_{r\times c}(K_{rv})$ with $t$ resolution classes.
\end{thm}
The third construction is explained as follows. 
\begin{example}~\label{36}
	There exists the following resolvable {\it $3\times 3$ grid-block packing} in \cite{7}.
	Now, we construct some Steiner trades from it.\\
	
	{\bf{resolvable $P_{3\times 3}(K_{18})$:}}
	$$
	\begin{tabular}{|c|}
	\hline
	$0, 1, 2$\\
	$3, 4, 5$\\
	$6, 7, 8$\\
	\hline
	\end{tabular}
	\hspace{0.5cm}
	\begin{tabular}{|c|}
	\hline
	$0, 4, 8$\\
	 $5, 6, 9$\\
	  $7, a, e$\\
	  \hline
	  \end{tabular}
	  \hspace{0.5cm}
	  	\begin{tabular}{|c|}
	  	\hline
	  	 $0, 9, d$\\
	  $a, h, 8$\\
	  	$c, 7, 3$\\
	  	\hline
	  	\end{tabular}
	  $$
	  $$
	  		\begin{tabular}{|c|}
	  		\hline
	  		$9, a, b$\\
	  		$c, d, e$\\
	  		$f, g, h$\\
	  		\hline
	  		\end{tabular}
	  		\hspace{0.5cm}
	  			\begin{tabular}{|c|}
	  			\hline
	  			$1, 3, f$\\
	  			$c, g, b$\\
	  			$h, 2, d$\\
	  			\hline
	  			\end{tabular}
	  			\hspace{0.5cm}
	  				\begin{tabular}{|c|}
	  				\hline
	  				$1, 5, b$\\
	  			$6, f, 2$\\
	  			$g, e, 4$\\
	  				\hline
	  				\end{tabular}
	$$
	\vspace{0.5cm}
	
	If we consider the rows of arrays of this resolvable $P_{3\times 3}(K_{18})$, then it is a 3-way 1-solely balanced set for $k=3$ of volume 6.
	If we consider the rows and columns of arrays of this resolvable $P_{3\times 3}(K_{18})$, then it is a 6-way 1-solely balanced set of volume 6 with same parameter.
	We can do this because this $P_{3\times 3}(K_{18})$ is resolvable and all elements of 
	$V(K_{18})$ appeared in the columns of $S_1$. Notice that the columns are written respectively 
	from up to down. For example, in the following 
	6-way 1-solely balanced set of volume 6. The blocks of $S_i$ are the columns of $S_{i-3}\ (4\leq i)$.  \\
	
	$$
	\begin{tabular}{c|c|c|c|c|c}
	$S_1$ & $S_2$ & $S_3$ & $S_4$ & $S_5$ & $S_6$\\
	\hline
	$012$ & $048$ & $09d$ & $036$ & $057$ & $0ac$\\
	$345$ & $569$ & $ah8$ & $147$ & $46a$ & $9hc$\\
	$678$ & $7ae$ & $c73$ & $258$ & $89e$ & $d83$\\
	$9ab$ & $13f$ & $15b$ & $9cf$ & $1ch$ & $16g$\\
	$cde$ & $cgb$ & $6f2$ & $adg$ & $3g2$ & $5fe$\\
	$fgh$ & $h2d$ & $ge4$ & $beh$ & $fbd$ & $b24$
	\end{tabular}
	$$ 
	\\
	Now, we apply Proposition~\ref{2} and obtain a $\mu$-way $(18+\mu, 4, 2)$ Steiner trade of volume $6\mu$
	for $\mu \in \{2, 3, 4, 5, 6\}$.
\end{example}
By the method of this example, we have the following theorem.
\begin{thm}~\label{34}\\
	i) If there exists a resolvable $P_{c\times c}(K_n)$ with $t$ resolution classes,
	then there exists a $\mu$-way $(n+\mu,c+1,2)$ Steiner trade of volume $\frac{n}{c}\mu$ for $\mu \leq 2t$.\\
	ii) If there exits a resolvable $P_{r\times c}(K_n)$ whit $t$ resolution classes,
		then there exists a $\mu$-way $(n+\mu,c+1,2)$ Steiner trade of volume $\frac{n}{c}\mu$ for $\mu \leq t$.
\end{thm}
\begin{proof}
	The proof is similar the previous example.
\end{proof}
\begin{cor}
	There exists a $\mu$-way $(q^n+\mu, q+1,2)$ Steiner trade of volume $q^{n-1}\mu$ for $\mu \leq 2\lfloor \frac{q^n-1}{2q-2} \rfloor$
\end{cor}
\begin{proof}
	The result is obtained by Theorem~\ref{31} and Theorem~\ref{34}.
\end{proof}

In Theorem~\ref{33}, we can replace a $\mu$-way 1-solely balanced set by resolvable packing $PA(v,c,1)$
and obtained the new resolvable $P_{r\times c}(K_{rv})$ with $\mu$ resolution classes.
The proof of Theorem~\ref{33} ( It is proved in \cite{7} ) is mentioned again in the following example.
But, we apply the 1-solely balanced set instead of the resolvable packing $PA(v,c,1)$.
\begin{example}
\rm{Consider the 1-solely balanced set that it is obtained in the previous example. 
Corresponding to each block of the one-solely balanced set an array is constructed as follows:\\
Each block $a, b, c$ change to the this block.\\
\begin{center}
\begin{tabular}{|c c c|}
	\hline
	$a_0$& $b_0$ & $c_0$\\
		$b_1$& $c_1$ & $a_1$\\
			$c_2$& $a_2$ & $b_2$\\
	\hline
\end{tabular}
\end{center}
Therefore, we have resolvable $P_{3\times 3}(K_{54})$ grid packing with six resolution classes.
The following arrays are the resolution class $R_1$ of resolvable $P_{3\times 3}(K_{54})$
which are constructed 
from the $S_1$.
\begin{center}
\begin{tabular}{| c c c |}
	\hline
	$0_0$& $1_0$ & $2_0$\\
		$1_1$& $2_1$ & $0_1$\\
			$2_2$& $0_2$ & $1_2$\\
	\hline
\end{tabular}
\hspace{0.5cm}
\begin{tabular}{| c c c |}
	\hline
	$3_0$& $4_0$ & $5_0$\\
		$4_1$& $5_1$ & $3_1$\\
			$5_2$& $3_2$ & $4_2$\\
	\hline
\end{tabular}
\hspace{0.5cm}
\begin{tabular}{| c c c |}
	\hline
	$6_0$& $7_0$ & $8_0$\\
		$7_1$& $8_1$ & $6_1$\\
			$8_2$& $6_2$ & $7_2$\\
	\hline
\end{tabular}\\
$\ $\\
$\ $\\
\begin{tabular}{| c c c |}
	\hline
	$9_0$& $a_0$ & $b_0$\\
		$a_1$& $b_1$ & $9_1$\\
			$b_2$& $9_2$ & $a_2$\\
	\hline
\end{tabular}
\hspace{0.5cm}
\begin{tabular}{| c c c |}
	\hline
	$c_0$& $d_0$ & $e_0$\\
		$d_1$& $e_1$ & $c_1$\\
			$e_2$& $c_2$ & $d_2$\\
	\hline
\end{tabular}
\hspace{0.5cm}
\begin{tabular}{| c c c |}
	\hline
	$f_0$& $g_0$ & $h_0$\\
		$g_1$& $h_1$ & $f_1$\\
			$h_2$& $f_2$ & $g_2$\\
	\hline
\end{tabular}
\end{center}
Now, we can apply the method of example~\ref{36} and construct the $6$-way 1-solely balanced set of volume 18
and 36 from this resolvable $P_{3\times 3}(K_{54})$.
However, The new $6$-way 1-solely balanced set can be obtain by adding three
6-way 1-solely balanced sets
of volume 6 on disjoint foundations. The 6-way 1-solely balanced set of volume 6 is 
obtained in Example~\ref{36}.} 
\end{example} 
According to this example, we have the following theorem.
\begin{thm}~\label{h}
Assume $r\leq c$. If there exists a $\mu$-way 1-solely balanced set with block size $c$
and foundation $v$, then there exists a resolvable $P_{r\times c}(K_{rv})$ with $\mu$ resolution classes.
\end{thm}
The characterization of $\mu$-way 1-solely balanced set with different parameters can be useful
for constructing the new resolvable $r\times c$ grid-block packing. For example, in the~\cite{7}  
it is proved that, there exists optimal $P_{2\times 2}(K_{n})$ for any $v\ \equiv\  0\ (\rm{mod}\ 4)$	 
in three pages. But, we can obtain this results by the 1-solely balanced set and Theorem~\ref{34} as follows.
\begin{thm}~\label{w1}
There exists optimal resolvable $P_{2\times 2}(K_{v})$ for any $v\ \equiv\  0\ (\rm{mod}\ 4)$	
\end{thm}
\begin{proof}
Each complete graph $K_{2m}$ has $2m-1$ 
one-factors. We consider this one factorization of $K_{2m}$
as a $(2m-1)$-way 1-solely balanced set of volume $m$ and block size 2.
Now apply Theorem~\ref{33}.
We use this $(2m-1)$-way 1-solely balanced set instead of 
a $PA(2v,2,1)$. Therefore, we have a resolvable $P_{2\times 2}(K_{4m})$.  
This resolvable $P_{2\times 2}(K_{4m})$ has $2m-1$ =$\lfloor \frac{4m-1}{2+2-2}\rfloor$
resolution classes. This number is optimal. 
Since $2m-1= \lfloor \frac{v-1}{r+c-2}\rfloor=\lfloor \frac{4m-1}{2+2-2}\rfloor$. 
 \end{proof}
 The method of Theorem~\ref{w1} is explained in the following example
 for $m=2$ and complete graph $K_4$.
 \begin{example}~\label{b}
The one-factorization of $K_4$ is a 3-way 1-solely balanced set.
\begin{center}
 \begin{tabular}{ c| c| c }
 	$S_1$& $S_2$ & $S_3$\\
 		\hline
 		$12$& $13$ & $14$\\
 			$34$& $24$ & $23$\\
 
 \end{tabular}	
 \end{center}
 Now, we construct the following optimal resolvable
 $P_{2\times 2}(K_{8})$ by the method of Theorem~\ref{33}.
\begin{center}
{\bf{$R_1$:}}
\begin{tabular}{| c c  |}
	\hline
	$1_0$& $2_0$\\
		$2_1$& $1_1$\\
	\hline
\end{tabular}
\hspace{0.5cm}
\begin{tabular}{| c c  |}
	\hline
	$3_0$& $4_0$\\
		$4_1$& $3_1$\\
	\hline
\end{tabular}
\hspace{0.5cm}
{\bf${R_2}$:}
\begin{tabular}{| c c  |}
	\hline
	$1_0$& $3_0$\\
		$3_1$& $1_1$\\
	\hline
\end{tabular}
\hspace{0.5cm}
\begin{tabular}{| c c  |}
	\hline
	$2_0$& $4_0$\\
		$4_1$& $2_1$\\
	\hline
\end{tabular}
\hspace{0.5cm}
{\bf{$R_3$:}}
\begin{tabular}{| c c  |}
	\hline
	$1_0$& $4_0$\\
		$4_1$& $1_1$\\
	\hline
\end{tabular}
\hspace{0.5cm}
\begin{tabular}{| c c  |}
	\hline
	$2_0$& $3_0$\\
		$3_1$& $2_1$\\
	\hline
\end{tabular}
\end{center}
 \end{example}
Also we can apply the $\mu$-way 1-solely balanced set and Theorem~\ref{32}.
Next we obtain some new $\mu'$-way 1-solely balanced set for $\mu'\geq \mu$. This result 
is interesting. Since, we can have new $\mu'$-way steiner trades. It is explained in the next 
example and Theorem.
\begin{thm}
	If there exists a $\mu$-way 1-solely balanced set with block size $c$ and foundation $v$ 
	and a resolvable $P_{r\times c}(K_{rc})$ with $s+1$ grid blocks, then there exists a 
	 $\mu'$-way $(rv+\mu',c+1,2)$ Steiner trade for $\mu'\leq s\mu+1$ of volume $\frac{rv}{c}\mu'$.
\end{thm} 
\begin{proof}
	The proof has resulted from Theorems~\ref{32} and~\ref{34}.
\end{proof}
This theorem can be applied for constructing the Steiner trades. For example, we know that  
		There exists a $(2v-1)$-way 1-solely balanced set with block size $2$ and foundation $v$ 
	and an optimal resolvable $P_{2\times 2}(K_{8})$ with $5+1$ grid blocks, then there exists a 
	$\mu'$-way $(2v+\mu',2+1,2)$ Steiner trade for $\mu'\leq 5(2v-1)+1=10v-4$ of volume $v\mu'$.
Also if we apply $\mu$-way 1-solely balanced set for block size $k\geq3$, then we obtain 
the other Steiner trades for block size $k\geq 4$.\\
Also, there exist two theorems for achieving a new $\mu$-way 1-solely balanced set
(resolvable $r\times c$ grid-block packing) from another by applying the 
resolvable $r\times c$ grid-block packings ($\mu$-way 1-solely balanced sets).
\begin{thm}
If there exists a $\mu$-way 1-solely balanced set with foundation size $v$ and block size $c$,
then there exists a $\mu'$-way 1-solely balanced set with foundation size $c\times v$, block size $c$ and volume $v$
for $\mu'\leq 2\mu$.
\end{thm} 
\begin{proof}
	By Theorem~\ref{h}, there exists a resolvable $P_{r\times c}(K_{rv})$ with $\mu$ resolution classes. Now, apply Theorem~\ref{34}
	and obtain the result.
\end{proof}
\begin{example}
	In Example~\ref{b} a $P_{2\times 2}(K_{8})$ is constructed from a 3-way 1-solely balanced set with  foundation size 4 and block size $2$.
	Now, we can construct a 6-way 1-solely balanced set from this $P_{2\times 2}(K_{8})$ of volume four as follows.
		
		$$
		\begin{tabular}{c|c|c|c|c|c}
		$S_1$ & $S_2$ & $S_3$ & $S_4$ & $S_5$ & $S_6$\\
		\hline
		$1_02_0$ & $1_03_0$ & $1_04_0$ & $1_02_1$ & $1_03_1$ & $1_04_1$\\
		$2_11_1$ & $3_11_1$ & $4_11_1$ & $2_01_1$ & $3_01_1$ & $4_01_1$\\
		$3_04_0$ & $2_04_0$ & $2_03_0$ & $3_04_1$ & $2_04_1$ & $2_03_1$\\
		$4_13_1$ & $4_12_1$ & $3_12_1$ & $4_03_1$ & $4_02_1$ & $3_02_1$
		\end{tabular}
		$$ 
\end{example}  
The second theorem handles with a  resolvable $P_{r\times c}(K_{v})$.  
\begin{thm}
	If there exists a resolvable $P_{c\times c}(K_{v})$ with $t$ resolution classes, 
	then there exists a resolvable $P_{r\times c}(K_{rv})$ with $2t$ resolution classes for $r\leq c$.
\end{thm} 
\begin{proof}
By Theorem~\ref{34}, there exists a $\mu$-way 1-solely balanced set with block size $c$, foundation size $v$ and
volume $\frac{v}{c}$ for $\mu\leq 2t$. Now apply Theorem~\ref{h}
and obtain the result.	
\end{proof} 

%%%%%%%%%%%%%%%%%%%%%%%%%%%%%%%%%%%%%%%%%%%%%%%%%%%%%%%%%%%%%%%%%%%%%%%%%%%%%%%%%%%%%%%%%%%%%
%-----------------------------------------------------------------------------------------------------------------------------------------------------------------------------------------------------------------------------------------------
\section{Completion of spectrum }
In~\cite{0}, we prove that there exists a 3-way $(v,k,2)$ Steiner trade of
volume $m$ when $m\geq12(k-1)$ for $k\geq 15$. In this section, we show it is correct also for $k\leq 14$.
For $k=1$, there does not exist any Steiner trade and for
$k=2$, we have the trivial case.
For $k=3$ and $k=4$, it is proved in~\cite{3}.
\subsection{Block size 5}
For $k=5$, we have five parts as follows:\\
1- By Theorem~\ref{4}, $S_{3s}(2,k)\subseteq \mathcal{N}\setminus\{1,\ldots,3k-4\}$ for $k\ne 4$.
Therefore, $S_{3s}(2,5)\subseteq \mathcal{N}\setminus\{1,\ldots,11\}$ and
there does not exist any 3-way $(v,5,2)$ Steiner trade 
of volume $m$ for $m\leq 11$.\\
2- By the third part of Theorem~\ref{4}, there exists 3-way $(v,5,2)$ Steiner trade of volume $\{12,15,18,\\21,24,27,30,33,36,39,\ldots\}=\{m:m=3l,\ l\geq 4\}$.\\
3- By Theorem~\ref{3}, there exists 3-way $(v,5,2)$ Steiner trade of volume $m\in\{32,33,34,35\}=\{m:m=9\times4-r,\ 0\leq r\leq 4\}$.\\
4- By the fourth part of Theorem~\ref{4}, there exists 3-way $(v,5,2)$ Steiner trade of volume $m$ for
$m\in\{16,20,24,28,40,\ldots\}=\{m:m=4l,\ l\geq 4\}$.\\
5- There exist 3-way $(v,5,2)$ Steiner trades of volume $m\in\{15, 20, 25\}$ by Construction 1. 
\begin{thm}
	There exists a 3-way $(v,5,2)$ Steiner trade of volume $m\geq 12$ except possibly when $m\in\{13,14, 17,19,22,23,26,29\}$.	
\end{thm}
\begin{proof}
	If there exist two 3-way Steiner trades $T$ and $T^*$ of volume $m_1$ and $m_2$ with disjoint foundations and same block size, then there exists a 3-way Steiner trade $T+T^*$ of volume $m_1+m_2$.
	By this method, we have the existence of a 3-way $(v,5,2)$ Steiner trade of the volumes $\{31, 41, 38, 46, 43, 37, 47, 48,49\}$ and volumes
	$m=20+m'$, for $m'\geq 30$.\\	
By previous description the remainder volumes are $\{13,14, 17,19,22,23,26,29\}$.
These volumes are prime numbers or products of pairs of prime numbers less than 30.
\end{proof}
\subsection{Block size 6}
For $k=6$, we have five parts as follows:\\
1- By Theorem~\ref{4}, $S_{3s}(2,6)\subseteq \mathcal{N}\setminus\{1,\ldots,14\}$ and
there does not exist any 3-way $(v,6,2)$ Steiner trade 
of volume $m$ for $m\leq 14$.\\
2- By the second part of Theorem~\ref{4}, there exists 
a 3-way $(v,6,2)$ Steiner trade of volume $m\in\{15,18,21,\dots\}=\{m:m=3l,\ l\geq5\}$.\\
3-By Theorem~\ref{3}, there exists 
a 3-way $(v,6,2)$ Steiner trade of volume 
$m\in\{40,41,42,43,44\}=\{m:m=9\times 5-r,\ 1\leq r\leq5\}$.\\
4- By the fourth part of Theorem~\ref{4}, there exists 
a 3-way $(v,6,2)$ Steiner trade of volume $m\in\{40,44,48,\dots\}=\{m:m=4l,\ l\geq10\}$.\\
5- By the fifth part of Theorem~\ref{4}, there exists a $(v,6,2)$ 3-way Steiner trade
of volume $m\in\{15,20,25,30\}$.\\
Now, we can state the following theorem.
\begin{thm}
	There exists a 3-way $(v,6,2)$ Steiner trade of volume $m\geq 15$ except possibly when
	  $m\in\{16,17,19,22,23,26,28,29,31,32,34,37\}$.	
\end{thm}
\begin{proof}
	Some volumes are multiples of three or four with the 
	conditions of Theorem~\ref{4}. Other volumes can be written as $m_1+m_2$ from previous parts. By this method, we have the existence of a 3-way $(v,6,2)$ Steiner trade of volumes$\{35,38,46,47,49,50,53,55,56,58,59,61,62,65,67,68,70,\\71,73,74\}$ 
and 
	$m=15+m'$, for $m'\geq 60$. There does not exist any 3-way $(v,6,2)$ Steiner trade of volume $m\leq 14$.
	Therefore, the remainder volumes are  $\{16,17,19,22,23,26,28,29,31,32,34,37\}$.
\end{proof} 
Some values of volumes can be obtained by other ways such as by using Construction 1.  
\begin{example}
	There exists a $RB(45,5,1)$. Therefore, by Construction 1, there exist $\mu$-way $(45+\mu, 6,2)$
	Steiner trades of volumes $\mu\frac{45}{5}=9\mu$, for $3\leq \mu\leq 11$. \\
	Steiner trade of volume $36-6=30$.\\
	There exists a $RB(25,5,1)$. By Construction 1, there exist  $\mu$-way $(25+\mu, 6,2)$
	Steiner trades of volumes $\mu\frac{25}{5}=5\mu$, for $3\leq \mu\leq 6$. Therefore, there exists a
	3-way $(v,6,2)$ Steiner trade of volume
	$m\in\{15,20,25,30\}$. \\
	\end{example}
\subsection{Block size 7}
For $k=7$, we have four parts as follows:\\
1- By Theorem~\ref{4}, $S_{3s}(2,7)\subseteq \mathcal{N}\setminus\{1,\ldots,17\}$ and
there does not exist any 3-way $(v,7,2)$ Steiner trade 
of volume $m$ for $m\leq 17$.\\
2- By the second part of Theorem~\ref{4}, there exists 
a 3-way $(v,7,2)$ Steiner trade of volume $m\in\{18,21,\ldots\}=\{m:m=3l,\ l\geq6\}$.\\
3- By Theorem~\ref{3}, there exists a 3-way $(v,7,2)$ Steiner trade of volume
$m\in\{47,48,49,50,51,52,\\53\}=\{m:m=9\times 6-r,\ 1\leq r\leq6\}$.\\
4- By the fourth part of Theorem~\ref{4}, there exists a 3-way $(v,7,2)$ Steiner trade of volume $m\in\{48,\ldots\}=\{m:m=4l,\ l\geq12\}$.\\
Now, we can state the following theorem.
\begin{thm}
	There exists a 3-way $(v,7,2)$ Steiner trade of volume $m\geq 18$ except possibly when
	  $m\in\{19,20,22,23,25,26,28,29,31,32,34,35,37,38,40,41,43,44,46,47,55,58,59,61,62,65\}$.	
\end{thm}
\begin{proof}
Some volumes are multiples of three or four with the 
conditions of Theorem~\ref{4}. Other volumes can be written as $m_1+m_2$ from previous parts. By this method, we have the existence of a 3-way $(v,7,2)$ Steiner trade of volumes
$\{57, 63, 67, 70, 71, 73,
74, 75, 76, 77, 79, 80,  
82, 83, 85, 86, 88,\\ 89, 91\}$ 
and
$m=18+m'$, for $m'\geq 72$.
There does not exist any 3-way $(v,7,2)$ Steiner trade of volume $m\leq 17$.	
		Therefore, the reminder volumes are $\{19,20,22,23,25,26,28,29,31,32,34,35,37,38,\\40,41,43,44,46,47,55,58,59,61,62,65\}$.
\end{proof} 
\begin{remark}
There exists a $RB(36,6,1)$. Therefore, by Construction 1 there exist $\mu$-way $(36+\mu, 7,2)$
Steiner trades of volumes $\mu\frac{36}{6}=6\mu$, for $3\leq \mu\leq \frac{35}{5}=7$.
\end{remark}
\subsection{Block size 8}
For $k=8$, we have five parts as follows:\\
1- By Theorem~\ref{4}, $S_{3s}(2,8)\subseteq \mathcal{N}\setminus\{1,\ldots,20\}$ and
there does not exist any 3-way $(v,8,2)$ Steiner trade 
of volume $m$ for $m\leq 20$.\\
2- By from the second part of Theorem~\ref{4}, there exists 
a 3-way $(v,8,2)$ Steiner trade of volume $m\in\{21,24,\ldots\}=\{m:m=3l,\ l\geq7\}$.\\
3- By Theorem~\ref{3}, there exists 
a 3-way $(v,8,2)$ Steiner trade of volume 
$m\in\{56,\ldots,62\}=\{m: m=9\times 7-r,\ 1\leq r\leq7\}$.\\
4- By the fourth part of Theorem~\ref{4}, there exists 
a 3-way $(v,8,2)$ Steiner trade of volume $m\in\{56,60,68,\ldots\}=\{m:m=4l,\ l\geq14\}\cup\{28\}\ (28=4\times7)$.\\
5- By the fifth part of Theorem~\ref{4}, there exists a 3-way $(v,8,2)$ Steiner trade
of volume $m\in\{39,52,65,78,91,104,117,130,143,156,169,182\}$.\\
Now, we can state the following theorem.
\begin{thm}
	There exists a 3-way $(v,8,2)$ Steiner trade of volume $m\geq 21$ except possibly
	  when $m\in\{22,23,25,26,29,30,31,34,37,38,
	  40,41,43,44,46,47,50,53\}$.	
\end{thm}
\begin{proof}
Some volumes are multiples of three or four with the 
conditions of Theorem~\ref{4}. Other volumes can be written as $m_1+m_2$ from previous parts. By this method, we have the existence of a 3-way $(v,8,2)$ Steiner trade of volumes 
$\{52, 55, 64, 67, 68, 70,
71, 73, 74,76,77,79,80,82,83,
85, 86,\\ 89, 91, 94, 95,
97, 98, 101, 103\}$ 
and
	$m=21+m'$, for $m'\geq 84$. There does not exist any 3-way $(v,8,2)$ Steiner trade of volume $m\leq 20$. 
	Therefore, the remainder volumes that we don't know the existence or non-existence of them are
	$\{22,23,25,26,29,30,31,3437,38,
	40,41,43,44,46,47,50,\\53\}$.	
\end{proof} 
\begin{remark} 
	{\rm{There exists a $RB(49,7,1)$. Therefore, by Construction 1 
			there exist 3-way $(v,8,2)$ Steiner trades of volumes
			$\frac{49}{7}\mu=7\mu$ for $3\leq\mu\leq\frac{48}{6}=8$.
			Therefore, there exists a 3-way $(v,8,2)$ Steiner trade of volume
			$m\in\{21,28,35,42,49,56\}$. }}
\end{remark}
\subsection{Block size 9}
For $k=9$, we have four parts as follows:\\
1- By Theorem~\ref{4}, $S_{3s}(2,9)\subseteq \mathcal{N}\setminus\{1,\ldots,23\}$ and
there does not exist any 3-way $(v,9,2)$ Steiner trade 
of volume $m$ for $m\leq 23$.\\
2- By the second part of Theorem~\ref{4}, there exists 
a 3-way $(v,9,2)$ Steiner trade of volume $m\in\{24,27,30,\ldots\}=\{m:m=3l,\ l\geq8\}$.\\
3- By Theorem~\ref{3}, there exists 
a 3-way $(v,9,2)$ Steiner trade of volume
$m\in\{71,70,\ldots,65,64\}=\{m:m=9\times 8-r,\ 1\leq r\leq8\}$.\\
4- By the fourth part of Theorem~\ref{4}, there exists 
a 3-way $(v,9,2)$ Steiner trade of volume $m\in\{64,68,72\ldots\}=\{m:m=4l,\ l\geq16\}$.\\
Now, we can state the following theorem.
\begin{thm}
		There exists a 3-way $(v,9,2)$ Steiner trade of volume $m\geq 24$ except possibly
		 when $m\in\{25,26,28,29,31,32,34,35,37,38,40,41,42,43
		  ,44,46,47,49,50,52,53,55,56,58,59,61,62,\\
		  73,74,77,79,82,83,85,86,\}$	
\end{thm}
\begin{proof}
Some volumes are multiples of three or four with the 
conditions of Theorem~\ref{4}. Other volumes can be written as $m_1+m_2$ from previous parts. By this method, we have the existence of a 3-way $(v,9,2)$ Steiner trade of volumes $\{88, 89, 94, 95, 97, 98, 101, 103, 106, 107,
109, 110,113, 115,\\ 118, 119\}$ 
and 
	$m=24+m'$, for $m'\geq 96$. There does not exist any 3-way $(v,9,2)$ Steiner trade of volume $m\leq 38$. 
Therefore, the remainder volumes that we don't know the existence or non-existence of them are
$\{25,26,28,29,31,32,34,35,37,38,40,41,42,43
,44,46,47,49,50,52,53,55,\\56,58,59,61,62,
73,74,77,79,82,83,85,86,\}$ 
\end{proof} 
\subsection{Block size 10}
For $k=10$, we have five parts as follows:\\
1- By Theorem~\ref{4},
$S_{3s}(2,10)\subseteq \mathcal{N}\setminus\{1,\ldots,26\}$ and
there does not exist any 3-way $(v,10,2)$ Steiner trade 
of volume $m$ for $m\leq 26$.\\
2- By the second part of Theorem~\ref{4}, there exists 
a 3-way $(v,10,2)$ Steiner trade of volume $m\in\{27,30,33\ldots\}=\{m:m=3l,\ l\geq9\}$.\\
3- By Theorem~\ref{3}, there exists 
a 3-way $(v,10,2)$ Steiner trade of volume 
$m\in\{80,79\ldots73,72\}=\{m:m=9\times 9-r,\ 1\leq r\leq9\}$.\\
4- By the fourth part of Theorem~\ref{4}, there exists 
a 3-way $(v,10,2)$ Steiner trade of volume $m\in\{72,76,\ldots\}=\{m:m=4l,\ l\geq18\}\cup\{36\}\ (36=4\times9)$.\\
5- There exists a $RB(153,9,1)$. Therefore, by Construction 1, there exist 3-way $(v,10,2)$ Steiner trades of volumes
$\frac{153}{9}\mu=17\mu$ for $3\leq\mu\leq\frac{152}{8}=19$. Therefore, there exists a 3-way $(v,10,2)$ Steiner trade of volume
$m\in\{68,85,102,119,\ldots\}$. The existence of a 3-way $(v,10,2)$ Steiner trade of volumes 68 and 85 is new.\\
Now, we can state the following theorem.
\begin{thm}
	There exists a 3-way $(v,10,2)$ Steiner trade of volume $m\geq 27$ except possibly when
	$m\in\{28,29,31,32, 34,35, 37, 38, 40,41,43,44,
	46,47,49,50,52,53,55,56,58,59,61,62,
	64,65,\\67,70,71,82,83,86,
	91,89,94,97\}$. 	
\end{thm}
\begin{proof}
Some volumes are multiples of three or four with the 
conditions of Theorem~\ref{4}. Other volumes can be written as $m_1+m_2$ from previous parts. By this method, we have the existence of a 3-way $(v,10,2)$ Steiner trade of volumes $\{95, 98,100,101,103,104,106,107
109, 110, 113, 115, 118,\\ 119, 121,
122, 125, 127, 130, 131, 133, 134, 135\}$ 
and 
	$m=39+m'$, for $m'\geq 156$. There does not exist any 3-way $(v,10,2)$ Steiner trade of volume $m\leq 26$. 
	Therefore, the remainder volumes that we don't know the existence or non-existence of them are
	$\{28,29,31,32, 34,35, 37, 38, 40,41,43,44,\\
	46,47,49,50,52,53,55,56,58,59,61,62,
	64,65,67,70,71,82,83,86,
	91,89,94,97\}$. 	
\end{proof} 
\subsection{Block size 11}
For $k=11$, we have four part as follows:\\
1- By Theorem~\ref{4}, $S_{3s}(2,11)\subseteq \mathcal{N}\setminus\{1,\ldots,29$ and
there does not exist any 3-way $(v,11,2)$ Steiner trade 
of volume $m$ for $m\leq 29$.\\
2- By the second part of Theorem~\ref{4}, there exists 
a 3-way $(v,11,2)$ Steiner trade of volume $m\in\{30,33,36,\ldots\}=\{m:m=3l,\ l\geq10\}$.\\
3- By Theorem~\ref{3}, there exists 
a 3-way $(v,11,2)$ Steiner trade of volume 
$m\in\{89,88,\ldots,80\}=\{m|m=9\times 10-r,\ 1\leq r\leq10\}$.\\
4- By the fourth part of Theorem~\ref{4}, there exists 
a 3-way $(v,11,2)$ Steiner trade of volume $m\in\{80,84,\ldots\}=\{m:m=4l,\ l\geq20\}$.\\
Now, we can state the following theorem.
\begin{thm}
	There exists a 3-way $(v,11,2)$ Steiner trade of volume $m\geq30$ except possibly
	when $m\in\{31,32,34,35,37,38,41,43,44,46,47,49,52,
	53,55,56,58,59,61,62,64,65,67,68,71,74,\\77\}$.	
\end{thm}
\begin{proof}
Some volumes are multiples of three or four with the 
conditions of Theorem~\ref{4}. Other volumes can be written as $m_1+m_2$ from previous parts. By this method, we have the existence of a 3-way $(v,11,2)$ Steiner trade of volumes $\{73, 79, 91, 92, 94, 97, 98, 101
103, 106, 107, 109, 113, 115, \\118, 119,
121, 122, 125, 127, 130, 131, 133,
134, 137, 139, 142, 143, 145, 146,
149.\}$ 
and  
	$m=30+m'$, for $m'\geq 120$. There does not exist any 3-way $(v,11,2)$ Steiner trade of volume $m\leq 29$. 
	Therefore, the remainder volumes that we don't know the existence or non-existence of them are
	$\{31,32,34,35,37,38,41,43,44,46,47,49,52,
	53,55,56,58,59,61,62,64,\\65,67,68,71,74,77\}$.	
\end{proof} 
\begin{remark} 
{\rm{There exists a $RB(100,10,1)$ and $RB(190,10,1)$. By Construction 1, there exists 3-way $(v,11,2)$ Steiner trades of volumes
$\frac{100}{10}\mu=10\mu$ for $3\leq\mu\leq\frac{99}{9}=11$ and $\frac{190}{10}\mu=19\mu$ for $3\leq\mu\leq\frac{189}{9}=21$.
Therefore, there exists a 3-way $(v,11,2)$ Steiner trade of volume
$m\in\{76,95,133,30,40,50,60,70,80,90,100,110\}$.
 The existence of some of 3-way $(v,11,2)$ Steiner trades of these volumes
are known from~Theorem~\ref{4} too.}}
\end{remark}
\subsection{Block size 12}
For $k=12$, we have five parts as follows:\\
1- By Theorem~\ref{4}, $S_{3s}(2,12)\subseteq \mathcal{N}\setminus\{1,\ldots,32\}$
and there does not exist any 3-way $(v,12,2)$ Steiner trade 
of volume $m$ for $m\leq 32$.\\
2- By the second part of Theorem~\ref{4}, there exists 
a 3-way $(v,12,2)$ Steiner trade of volume $m\in\{33,36,39,42,45,\ldots\}=\{m:m=3l,\ l\geq11\}$.\\
3- By Theorem~\ref{3}, there exists 
a 3-way $(v,12,2)$ Steiner trade of volume
$m\in\{98,97,96,\ldots,89,\\88\}=\{m:m=9\times 11-r,\ 1\leq r\leq11\}$.\\
4- By the fourth part of Theorem~\ref{4}, there exists 
a 3-way $(v,12,2)$ Steiner trade of volume $m\in\{88,92,96,\dots\}=\{m:m=4l,\ l\geq26\}\cup\{44\}\ (44=4\times11)$.\\
5- By the fifth part of Theorem~\ref{4}, there exists a 3-way $(v,12,2)$ Steiner trade
of volume $m\in\{33,44,55,66,77,88,99,110,132\}$.\\
Now, we can state the following theorem.
\begin{thm}
		There exists a 3-way $(v,12,2)$ Steiner trade of volume $m\geq 33$ except possibly when
		$m\in\{34,35,37,38,40,41,43,46,47,49,50,52,53,56,58,59,61,62,
		64,67,68,70,71,73,74,76,\\85\}$. 	
\end{thm}
\begin{proof}
Some volumes are multiples of three or four with the 
conditions of Theorem~\ref{4}. Other volumes can be written as $m_1+m_2$ from previous parts. By this method, we have the existence of a 3-way $(v,12,2)$ Steiner trade of volumes$\{86, 101, 103, 118, 119, 121, 122, 124,125,
127, 128,130,131,\\
133, 134, 137, 139, 142, 143, 145,
146, 149, 151, 154, 155, 157, 158,
161, 163\}$ 
and 
$m=33+m'$, for $m'\geq 132$. There does not exist any 3-way $(v,12,2)$ Steiner trade of volume $m\leq 32$. 
Therefore, the remainder volumes that we don't know the existence or non-existence of them are
	$\{34,35,37,38,40,41,43,46,47,49,50,52,53,56,58,59,61,62,
	64,67,68,70,\\71,73,74,76,85\}$. 
\end{proof}
%\begin{remark} 
%{????\rm{There exist a $RB(132,11,1)$. So by Theorem~\ref{110} there exist a 3-way $(v,12,2)$ Steiner trade of volume
%$\frac{132}{11}\mu=12\mu$ for $3\leq\mu\leq\frac{131}{10}=$. So, there exist a 3-way $(v,14,2)$ Steiner trade of volume
%$s\in\{52,65,91,104,130,143,169,182\}$. The existence of a 3-way $(v,14,2)$ Steiner trade of these volumes
%is known from~Theorem~\ref{4} too.}}????
%\end{remark}
\subsection{Block size 13}
For $k=13$, we have four parts as follows:\\
1- By Theorem~\ref{4}, $S_{3s}(2,13)\subseteq \mathcal{N}\setminus\{1,\ldots,35\}$ and
there does not exist any 3-way $(v,13,2)$ Steiner trade 
of volume $m$ for $m\leq 35$.\\
2- By the second part of Theorem~\ref{4}, there exists 
a 3-way $(v,13,2)$ Steiner trade of volume $m\in\{36,39,42,\ldots\}=\{m:m=3l,\ l\geq12\}$.\\
3- By Theorem~\ref{3}, there exists 
a 3-way $(v,13,2)$ Steiner trade of volume
$m\in\{107,106,105,\ldots,97,\\96\}=\{m:m=9\times 12-r,\ 1\leq r\leq12\}$.\\
4- By the fourth part of Theorem~\ref{4}, there exists 
a 3-way $(v,13,2)$ Steiner trade of volume $m\in\{96,100,\ldots\}=\{m:m=4l,\ l\geq24\}$.\\
Now, we can state the following theorem.
\begin{thm}
		There exists a 3-way $(v,13,2)$ Steiner trade of volume $m\geq 36$ except possibly
		$m\in\{37,38,40,41,43,44,46,47,49,50,53,55,56,58,59,61,62,
		64,67,68,70,71,73,74,76,77,79,80,\\82,83,85,86,
		89,92,95\}$.	
\end{thm}
\begin{proof}
Some volumes are multiples of three or four with the 
conditions of Theorem~\ref{4}. Other volumes can be written as $m_1+m_2$ from previous parts. By this method, there exists a 3-way $(v,13,2)$ Steiner trade of volumes$\{88, 94, 118, 119, 122, 125, 127, 131,
133, 134,137, 139, 142, 142, 143,\\
145, 146, 149, 151, 154, 155, 157,
158, 161, 163, 166, 167, 170, 171,
173, 175\}$ and 
$m=36+m'$, for $m'\geq 144$. There does not exist any 3-way $(v,13,2)$ Steiner trade of volume $m\leq 35$. 
Therefore, the remainder volumes are $\{37,38,40,41,43,44,46,47,49,50,53,55,56,58,59,61,62,
64,67,68,70,\\71,73,74,76,77,79,80,82,83,85,86,
89,92,95\}$.
\end{proof}
\begin{remark} 
{\rm{There exists a $RB(144,12,1)$. By Construction 1, there exists a 3-way $(v,13,2)$ Steiner trade of volume
$\frac{144}{12}\mu=12\mu$ for $3\leq\mu\leq\frac{143}{11}=13$. The existence of  3-way $(v,13,2)$ Steiner trades of these volumes
are known from~Theorem~\ref{4} too.}}
\end{remark}
\subsection{Block size 14}
For $k=14$, we have five parts as follows:\\
1- By Theorem~\ref{4}, $S_{3s}(2,14)\subseteq \mathcal{N}\setminus\{1,\ldots,38\}$
and there does not exist any 3-way $(v,14,2)$ Steiner trade 
of volume $m$ for $m\leq 38$.\\
2- By the second part of Theorem~\ref{4}, there exists 
a 3-way $(v,14,2)$ Steiner trade of volume $m\in\{39,42,45,\ldots\}=\{m:m=3l,\ l\geq13\}$.\\
3- By Theorem~\ref{3}, there exists 
a 3-way $(v,14,2)$ Steiner trade of volume
$m\in\{116,115,114,\ldots,\\105,104\}=\{m:m=9\times 13-r,\ 1\leq r\leq13\}$.\\
4- By the fourth part of Theorem~\ref{4}, there exists 
a 3-way $(v,14,2)$ Steiner trade of volume $m\in\{104,108,\ldots\}=\{m:m=4l,\  l\geq26\}\cup{52=4\times13}$.\\
5- By the fifth part of Theorem~\ref{4}, there exists a 3-way $(v,14,2)$ Steiner trade
of volume $m\in\{39,52,65,78,91,104,117,130,143,\\156,169,182\}$.\\
Now, we can state the following theorem.
\begin{thm}
	There exists a 3-way $(v,14,2)$ Steiner trade of volume $m\geq 39$ except possibly when 
	$m\in\{40,41,43,44,46,47,49,50,53,55,56,58,59,61,62,
	64,67,68,70,71,73,74,76,77,79,80,\\82,83,85,86,88,
	89,92,95,98,101,131\}$.	
\end{thm}
\begin{proof}
Some volumes are multiples of three or four with the 
conditions of Theorem~\ref{4}. Other volumes can be written as $m_1+m_2$ from previous parts.
By this method, there exists a 3-way $(v,14,2)$ Steiner trade of volumes$\{94, 97, 100, 103, 118, 119, 122, 125,
127, 133, 134, 137, 139, 142,143,\\145,146,148,149,151,152,154,155,
157, 158, 161, 163, 166, 167, 169,
170,	173, 175, 178, 179, 181, 182,\\ 
183, 185, 187, 190, 191, 193,194\}$ 
and volumes 
	$m=39+m'$, for $m'\geq 156$. There does not exist any 3-way $(v,14,2)$ Steiner trade of volume $m\leq 38$. 
Therefore, the remainder volumes that we don't know the existence or non-existence of them are
	$\{40,41,43,44,46,47,49,50,53,55,56,58,59,61,62,\\
	64,67,68,70,71,73,74,76,77,79,80,82,83,85,86,88,
	89,92,95,98,101,131\}$.	
	\end{proof}
\begin{remark} 
{\rm{There exists a $RB(169,13,1)$. By Construction 1 there exists a 3-way $(v,14,2)$ Steiner trade of volume
$\frac{169}{13}\mu=13\mu$ for $3\leq\mu\leq\frac{168}{12}=14$. Therefore, there exists a 3-way $(v,14,2)$ Steiner trade of volume
$m\in\{52,65,91,104,130,143,169,182\}$. The existence of a 3-way $(v,14,2)$ Steiner trade of these volumes
is known from~Theorem~\ref{4} too.}}
\end{remark}
%---------------------------------------------------------------------------------------------------------------------------------------------------------------------------------------------------------------------------------
%%%%%%%%%%%%%%%%%%%%%%%%%%%%%%%%%%%%%%%%%%%%%%%%%%%%%%%%%%%%%%%%%%%%%%%%%%%%%%%%%%%%%
%---------------------------------------------------------------------------------------------------------------------------------------------------------------------------------------------------------------------------------

%------------------------------------------------------------------------------------------------------------------------------------------------------------------------------------------------------------------------------------
%%%%%%%%%%%%%%%%%%%%%%%%%%%%%%%%%%%%%%%%%%%%%%%%%%%%%%%%%%%%%%%%%%%%%%%%%%%%%%%%%%%%%%%%%
%-------------------------------------------------------------------------------------------------------------------------------------------------------------------------------------------------------------------------------------
\providecommand{\bysame}{\leavevmode\hbox to3em{\hrulefill}\thinspace}
\providecommand{\MR}{\relax\ifhmode\unskip\space\fi MR }
\providecommand{\MRhref}[2]{%
\href{http://www.ams.org/mathscinet-getitem?mr=#1}{#2}
}
\providecommand{\href}[2]{#2}


\begin{thebibliography}{10}
	
\bibitem{61} H. Amjadi and N. Soltankhah, \emph{On the existence of $d$-homogeneous $3$-way Steiner trades},  Utilitas Math. (to appear).

  \bibitem{forbes}
  A. D. Forbes, M. J. Grannell and T. S. Griggs, \emph{Configurations and trades in Steiner triple systems}, Australas. J. Combin.  \textbf{29} (2004) 75-84.

\bibitem{1}
B.~D. Gray and C.~Ramsay, \emph{On the spectrum of {S}teiner {$(v,k,t)$} trades {I}}, J.
Combin. Math. Combin. Comput. \textbf{34} (2000), 133--158.
 \bibitem{5}
\bysame, \emph{On the spectrum of {S}teiner {$(v,k,t)$} trades {II}},
Graphs and Combin. \textbf{15} (1999), 405--415.

\bibitem{62} S. Golalizadeh and N. Soltankhah, \emph{The minimum possible volume size of $\mu$-way $(v,k,t)$ trades},  Utilitas Math. (to appear).

\bibitem{11}N. Hamilton and A. Khodkar, \emph{On minimum possible volumes of strong Steiner trades},
{Australas. J. Combin. }, \textbf{20}, (1999), 197--203.

 \bibitem{8}
A.~Khodkar and D.~G. Hoffman, \emph{On the non-existence of
{S}teiner {$(v,k,2)$} trades with certain volumes},
Australas. J. Combin. \textbf{18} (1998), 303--311.

\bibitem{6}  R. C. Mullin and H. D. O. F. Gronau, \emph{PBDs, Frames, and Resolvability in: CRC Handbook of Combinatorial Designs (CJ Colbourn, JH Dinitz; eds.)}, (1996) 224-226.

\bibitem{7}
Y. Mutoh,  M. Jimbo and  H.L. Fu, \emph{A resolvable r x c grid-block packing and its application to DNA library screening}, Taiwanese J. of Math. 
\textbf{8(4)} (2004), 713--737.


\bibitem{3}
S.~Rashidi and N.~Soltankhah, \emph{On the possible volume of {$\mu$-$(v,k,t)$} trades},
Bull. Iranian Math. Soc. \textbf{40} (2014), no.~6, 1387--1401.

\bibitem{0} 
\bysame, 
\emph{On the 3-way (v, k, 2) Steiner trades}. Discrete Math. \textbf{339(12)} (2016) 2955--2963. 

\bibitem{10}
F. Steven, M. Ying and Y. Jianxing, \emph{Frames and resolvable designs. Uses, constructions, and existence},
CRC Publishing Co., Boca Raton, Florida (1996)



 \end{thebibliography}
 \end{document}